\documentclass[a4paper, 12pt]{amsart}

\usepackage{amssymb}
\usepackage{amsmath}
\usepackage{amsthm}
\usepackage{amscd}

\usepackage{comment}
\usepackage[all]{xy}
\usepackage{color}

\title{Good coverings of Alexandrov spaces}

\author{Ayato Mitsuishi}
\email{{mitsuishi@fukuoka-u.ac.jp}}
\address{{Department of Applied Mathematics, Fukuoka University, Jyonan-ku, Fukuoka-shi, Fukuoka 814-0180, JAPAN}}

\author{Takao Yamaguchi}
\thanks{This work was supported by JSPS KAKENHI Grant Numbers 26287010, 15H05739, 15K17529}
\email{takao@math.kyoto-u.ac.jp}
\address{Department of mathematics, Kyoto University, Kitashirakawa, Kyoto 606--8502, JAPAN}

\theoremstyle{plain}
\newtheorem{theorem}{Theorem}[section]
\newtheorem{lemma}[theorem]{Lemma}

\newtheorem{proposition}[theorem]{Proposition}
\newtheorem{definition}[theorem]{Definition}
\newtheorem{remark}[theorem]{Remark}

\newtheorem{assertion}[theorem]{Assertion}

\makeatletter

\newcommand{\wangle}[0]{\tilde{\angle}}
\newcommand{\diam}[0]{\mathrm{diam}\,}

\newcommand{\e}[0]{\epsilon}

\newcommand{\vol}[0]{\mathrm{vol}}

\newcommand{\supp}[0]{\mathrm{supp}}

\newcommand{\Lip}[0]{\mathrm{Lip}}

\newcommand{\relmiddle}[1]{\mathrel{}\middle#1\mathrel{}}

\newcommand{\beq}[0]{\begin{equation}}
\newcommand{\eeq}[0]{\end{equation}}

\newcommand{\benum}[0]{\begin{enumerate}}
\newcommand{\eenum}[0]{\end{enumerate}}

\begin{document}
\begin{abstract}
In the present paper, we define a notion of good coverings of 
Alexandrov spaces with curvature bounded below, and prove that
every Alexandrov space admits such a good covering and that
it has the same homotopy type as the nerve of  {the} good covering.
We also prove the stability of the isomorphism classes of the nerves of good coverings 
in the non-collapsing case. In the proof, we need a version of Perelman's fibration 
theorem, which is also proved in this paper.
\end{abstract}

\maketitle

\section{Introduction}
It is well known that there are relations between coverings and topology
of spaces. In Riemannian geometry, Weinstein \cite{We} found homotopy type 
finiteness of even-dimensional closed Riemannian manifolds of
positively pinched curvature by covering those manifolds via convex balls whose number is 
uniformly bounded. Then Cheeger \cite{Ch} extended this result to 
diffeomorphism finiteness by using a gluing method
to a wider class of closed Riemannian manifolds with 
bounded sectional curvature. 
In the context of a lower sectional curvature bound,
Grove and Petersen \cite{GP} succeeded to have a uniform bound
on the number of metric balls, which 
are contractible {\it in} a larger concentric balls, 
needed to cover those Riemannian manifolds.
See also \cite{Yam:homo}, \cite{GPW}, \cite{Pets} for related results.
Covering methods are also useful to obtain bounds 
on the total Betti numbers. See \cite{Gr}, \cite{Yam:essen} for instance.

A covering of a topological space (resp. smooth manifolds) is 
called {\it good} if every nonempty finite intersection of elements 
in the covering is contractible (resp. diffeomorphic to an Euclidean space).
See for instance \cite{BT}.
In the present paper, we introduce a notion of 
good coverings of Alexandrov spaces with curvature bounded below.　

Let $M$ be an Alexandrov space with curvature bounded below.
An open set $U$ of $M$ is called 
{\it conical} and strongly Lipschitz contractible (SLC in short) 
if it is homeomorphic to the tangent cone 
at a point $p\in U$ and is strongly Lipschitz contractible 
to $p$ (see Sections \ref{sec:SLC} and \ref{sec:good-cov} for precise definitions).
An open set $U$ is called convex if every minimal geodesic segment 
joining any two points of $U$ is contained in $U$.

We say that a locally finite covering $\mathcal U=\{ U_i\}$ of $M$ is {\it good} if
every nonempty intersection $U_{i_0}\cap\dots\cap U_{i_m}$ is 
a convex, conical SLC bounded domain. 

The main results of the present paper are stated as follows. 

\begin{theorem}\label{thm:good-cover}
For every open covering $\mathcal V$ of an Alexandrov space $M$, 
\begin{enumerate}
\item there exists a locally finite refinement $\mathcal U$ of $\mathcal V$
which is a good covering. 
\item $M$ has the same homotopy type as the nerve of 
any good covering of it.
\end{enumerate}
\end{theorem}

Theorem \ref{thm:good-cover}(2) explicitly provides the homotopy type of any  
Alexandrov space from the information of a good covering.
%
%

Let $\mathcal A(n,D,v_0)$ denote the set of all isometry classes of 
$n$-dimensional compact Alexandrov spaces $M$ with
\[
 { \text{curvature} \ge -1,\,\,  \diam(M)\le D, \,\,  \vol(M)\ge v_0>0.}
\]
We have the following stability of the nerves of good coverings 
of Alexandrov spaces.

\begin{theorem} \label{thm:good-stab}
There exist a positive number $\e_0=\e_0(n,D,v_0)$, finitely many 
$M_1,\ldots, M_N \in \mathcal A(n,D,v_0)$ and finite simplicial complexes 
$K_1, \ldots, K_N$ such that 
\begin{enumerate}
\item $\mathcal A(n,D,v_0) = \bigcup_{i=1}^N U_{GH}(M_i,\epsilon_0);$
\item for any $M\in U_{GH}(M_i,\epsilon_0)$, there exists a good covering $\mathcal U_M$ of 
 $M$ whose nerve is isomorphic to $K_i$. 
\end{enumerate}
\end{theorem}

Here $U_{GH}(M,\epsilon)$ denotes the $\epsilon$-neighborhood of $M$ in the Gromov-Hausdorff distance.

\begin{remark} \upshape
Theorem \ref{thm:good-stab} is new even in the case {when} $M$ is a Riemannian manifold.
Together with Theorem \ref{thm:good-cover}, it enables us to compute the homotopy type of $M$
in terms of only the covering data of a good covering having only finite types.
\end{remark}

In the course of the proofs of Theorems \ref{thm:good-cover} and \ref{thm:good-stab}, we obtain
the following, which is needed in the proof of Theorem \ref{thm:good-cover}(1) to show the conical 
property in the conditions of good coverings.

\begin{theorem} \label{thm:CSLC}
Let $f:U\to \mathbb R$ be a proper strictly concave function defined on
 a connected open subset $U$ of an Alexandrov space $M$. Then 
\begin{enumerate}
\item there is a point $p \in U$ such that $\Omega:=\{ f\ge a\}$ is convex SLC to $p$ 
 for any $a$ with $\inf_{U} f < a< \max_{U} f$.
\item $\Omega$ is conical if either $\Omega$ does not meet $\partial M$, or 
 $\Omega$ meets $\partial M$ and the function $\tilde f:D(\Omega)\to \mathbb R$ 
 naturally induced by $f$ on the double $D(\Omega)$ of $\Omega$ is strictly 
 concave. 
\end{enumerate}
\end{theorem}

Here the double $D(\Omega)$ is defined as the disjoint union $\Omega \amalg \Omega$ glued along
their boundaries $\Omega\cap\partial M$.\par

In Theorem \ref{thm:CSLC}(2), we have counter examples if we drop the assumption on $\tilde f$.
It should also be remarked that in Theorem \ref{thm:CSLC}, the gradient flow of a strictly concave function $f$ 
might take infinite time to reach the unique maximum point of $f$ in general. 
Therefore the gradient flow of $f$ is not enough for the construction of a strong Lipschitz contraction, and 
we need additional 
arguments in the proof.

In the proof of Theorem \ref{thm:good-cover}(1), we also need to establish a version of 
Perelman's fibration theorem:

\begin{theorem}[cf. \cite{Per Alex II}, \cite{Per}, \cite{Per DC}] \label{thm:fibration}
Let $f:U\to \mathbb R$ be a proper semiconcave function defined on an open set $U$ 
of an Alexandrov space. If one of the following conditions holds, then $f$ is a fiber bundle over $f(U)$:
\begin{enumerate}
 \item if  $U$ does not meet $\partial M$, $f$ is regular on $U$,
 \item if  $U$ meets $\partial M$,  the canonical extension of $f$ to the double of $U$ is also 
          semiconcave and is  regular on it. 
\end{enumerate}
\end{theorem}

Theorem \ref{thm:fibration} was proved for admissible functions $f$ on $U$ possibly with boundary 
in \cite{Per Alex II} and \cite{Per}, for semiconcave functions on $U$ without boundary in \cite{Per DC}. 
Our contribution is in the case when $f$ is a semiconcave function and $U$ meets $\partial M$.


The organization of the paper is as follows. 
In Section \ref{sec:prelim}, we briefly recall several notions
about Alexandrov spaces, SLC neighborhoods and semiconcave
functions.
In Section \ref{sec:SLC}, we prove that a metric ball is 
SLC if the distance function from the center 
is regular on the ball. This extends a previous result in 
\cite{MY:SLC}. To achieve this, we develop a consecutive gluing method
of gradient flows of several distance functions by proving the Lipschitz
regularity of an implicit function.
Such a gluing procedure is turned out to be useful also in the proof of Theorem \ref{thm:CSLC}.
In Sections \ref{sec:good-cov} and 
\ref{sec:good-stab}, we prove Theorems \ref{thm:good-cover} 
and \ref{thm:good-stab} respectively by making use of Theorem \ref{thm:CSLC}.
Theorem \ref{thm:fibration} is proved in Section \ref{sec:appendix}.

\section{Preliminaries} \label{sec:prelim}

Let us recall the definition of Alexandrov spaces and related fundamental facts. 
For more details, we refer to \cite{BGP} and \cite{BBI}.
Throughout the present paper, we denote by $|xy|$ the distance between points $x$ and $y$ in a metric space.
{For a metric space $X$, $x \in X$ and $r > 0$, we denote by $U(x,r)$, $B(x,r)$ and $S(x,r)$, the open $r$-ball, the closed $r$-ball and the metric $r$-sphere around $x$, respectively.}

\subsection{Basics of Alexandrov spaces}
A metric space is said to be geodesic if any two points in the space can be joined by a minimal geodesic, where a minimal geodesic is an isometric embedding from an interval.

We say that a geodesic complete metric space $M$ is an {\it Alexandrov space} (of curvature bounded locally from below) if for each $p \in M$, there exist $r > 0$ and $\kappa \in \mathbb R$ such that for any distinct four points $a_i \in B(p,r)$, $i=0,1,2,3$ with $\max_{1 \le i < j \le 3} \{|a_0 a_i|+ |a_0 a_j|+ |a_i a_j|\} < \pi/\sqrt \kappa$ if $\kappa > 0$, we have 
\[
\sum_{1 \le i<j \le 3} \wangle_\kappa a_i a_0 a_j \le 2 \pi.
\]

Here, $\wangle_\kappa b a c$ denotes the inner angle of a geodesic triangle of length $|ab|$, $|bc|$ and $|ca|$, at the vertex with opposite side of length $|bc|$, in a simply-connected complete surface of curvature $\kappa$.
In the present paper, we only deal with finite-dimensional Alexandrov spaces.

From now on, let $M$ denote an $n$-dimensional Alexandrov space. 
For an Alexandrov space $M = (M,|\cdot, \cdot|)$ and $r > 0$, we denote by $r M$ the space $(M, r |\cdot, \cdot|)$. 
For $p \in M$, the pointed Gromov-Hausdorff limit of $(rM, p)$ as $r \to \infty$ always exists and is denoted by $(T_p M, o)$, which is called the {\it tangent cone} of $M$ at $p$.
An element of $T_p M$ is called a vector. 
For two vectors $v,w \in T_p M$, we set $\left<v,w \right> = |v||w| \cos \angle v o w$ if $|v| \neq 0 \neq |w|$ and $\left< v,w \right> = 0$ otherwise, where $|v|$ is the distance from $v$ to the origin $o$.

For $p \in M$, the set of all non-trivial unit-speed geodesic starting at $p$ is denoted by $\Sigma_p'$, which admits an equivalent relation defined by $\gamma \sim \sigma$ if and only if $\angle(\gamma, \sigma) = \lim_{s,t \to 0} \tilde \angle_\kappa \gamma(s) p \sigma(t) = 0$, for fixed $\kappa$.
The equivalent class of $\gamma$ is denoted by $\gamma^+(0)$, where $\gamma$ is assumed to be parametrized $\gamma(0) = p$. 
Then, $\angle$ is a metric on the set of all equivalent classes. 
The completion of it by $\angle$ is denoted by $\Sigma_p$ and is called the {\it space of directions} at $p$. 
Each element of $\Sigma_p$ is called a direction.
For $q \neq p$, we denote by $\uparrow_p^q \,\in \Sigma_p$ the direction of a minimal geodesic from $p$ to $q$ at $p$. 

The tangent cone $T_p M$ is isometric to the Euclidean cone over $\Sigma_p$.
So, any vector $v \in T_p M$ can be written as $v = a \xi$ for some $a \ge 0$ and $\xi \in \Sigma_p$.

For a Lipschitz curve $c : [0,a] \to M$, it has the direction at $t = 0$ if $\varlimsup_{s, t \to 0} \tilde\angle_\kappa c(s)c(0)c(t) = 0$ holds, for some fixed $\kappa$.
Then, the vector $c^+(0)$ is canonically defined as the limit of $|c(0)c(t)| \uparrow_{c(0)}^{c(t)}$ as $t \to 0$.

The boundary $\partial M$ is defiend as the set of all points $p \in M$ such that $\Sigma_p$ has non-empty boundary. 
Here, one-dimensional Alexandrov spaces are one-dimensional Riemannian manifolds possibly with boundary, whose boundaries are defiend as the boundaries of manifolds. 


\subsection{Strong Lipschitz contractibility}

\begin{definition}[\cite{MY:SLC}] \upshape
Let $X$ be a metric space, $p \in X$ and $r > 0$.
We say that a subset $\Omega$ of $X$ is {\it strongly Lipschitz contractible} (abbreviated by SLC) to some point $p \in \Omega$, 
if there is a map 
\[
H : \Omega \times [0,1] \to \Omega
\]
which is Lipschitz in the sense that 
\[
|H(x,s)H(y,t)| \le A|xy| + B|s-t|
\]
holds on the domain for some $A,B \ge 0$, such that $H_0(x)=x$, $H_1(x)=p$, and the distance
\[
d(H_t(x),p)
\]
is monotone non-increasing in $t$ for every $x \in \Omega$. 
Here, $H_t(x) = H(x,t)$.

For a subset $A \subset \Omega$, we say that $\Omega$ is {\it strongly Lipschitz contractible to $A$} 
if there is a Lipschitz map $H : \Omega \times [0,1] \to \Omega$ such that 
$H_0(x)=x$ and $H_1(x) \in A$ for every $x \in \Omega$, 
the function $d(H_t(y),A)$ is monotone non-increasing in $t$ for every $y \in \Omega$ and 
$H_t(z) = z$ for all $z \in A$ and $t \in [0,1]$.
\end{definition}
Note that if $B(p,r)$ is SLC to $p$, then $B(p,r')$ is also SLC to $p$ for every $r' < r$.

In \cite{MY:SLC}, we proved that every Alexandrov space is {\it strongly locally Lipschitz contractible} in the following sense. 
\begin{theorem}[\cite{MY:SLC}] \label{thm:SLCB}
Let $M$ be an Alexandrov space. 
For every $p \in M$, there is an $r > 0$ such that $B(p,r)$ is strongly Lipschitz contractible to $p$. 
\end{theorem}

\subsection{Semiconcave functions and their gradient flows}
Following \cite{Pet}, we recall the notion of the gradients of semiconcave functions on Alexandrov spaces and their properties. 

Let $M$ be an Alexandrov space. 
A locally Lipschitz function $f$ defined on an open subset $U$ of $M$ 
is said to be {\it semiconcave} if for any $x \in U$, there are $r > 0$ and $\lambda \in \mathbb R$ such that for any minimal geodesic $\gamma : [0,T] \to U(x,r)$ of unit speed contained in $U(x,r)$, the function $f \circ \gamma(t) - (\lambda/2)t^2$ is concave on $(0,T)$ in the usual sense.
In this case, $f$ is said to be $\lambda$-concave at $x$ and on $U(x,r)$. 
Let us set $\underline\lambda(x) = \inf\{\lambda \mid f \text{ is } \lambda \text{-concave at } x\}$. 
Then, $\underline\lambda$ is upper semicontinuous on $U$. 
Indeed, for any $\e > 0$, there is $r > 0$ such that $f$ is $(\underline\lambda(x)+\e)$-concave on $U(x,r)$. 
Then, $f$ is $(\underline\lambda(x)+\e)$-concave on $U_y(r-|xy|)$. 
Hence, we have $\varlimsup_{y \to x} \underline\lambda(y) \le \underline\lambda(x)$.
If a function $g : U \to \mathbb R$ satisfies $g(x) \ge \underline\lambda(x)$, we also say that $f$ is $g$-concave. 
We say that $f$ is {\it strictly concave} (concave, resp.) if $\underline{\lambda} < 0$ ($\le 0$, resp.) on the domain. 

The distance function from a closed set $A$ of an Alexandrov space $M$ is semiconcave on $M \setminus A$. 

Let $f$ be a semiconcave function defined on an open subset $U$ of an Alexandrov space $M$. 
For $x \in U$, we can define the {\it differential} $f' = f_x' : T_x M \to \mathbb R$ of $f$ at $x$ by 
\[
f_x'(c^+(0)) = \lim_{t \to 0} \frac{f(c(t))-f(c(0))}{t}
\]
for any curve $c : [0,a) \to U$ with $c(0) = x$ having the direction at $t = 0$. 
The map $f_x' : T_x M \to \mathbb{R}$ is a $0$-concave function. 

The {\it gradient} of $f$ at $x$ is the vector $\nabla_x f = \nabla f \in T_x M$ uniquely determined by the relations 
\[
|\nabla f|^2 = f'(\nabla f) \text{ and } \left< \nabla f, v \right> \ge f'(v)
\]
for every $v \in T_x M$. 
The {\it gradient curve} of $f$ is a curve $c : [0,a) \to M$ which has the direction at any time $t \in [0,a)$ and satisfies 
\[
c^+(t) = \nabla_{c(t)} f
\]
for every $t \in [0,a)$.

\begin{theorem}[\cite{PP}, \cite{Pet:multi}]
For any semiconcave function $f$ on an open subset $U$, and for any $x \in U$, there exists the unique maximal gradient curve starting at $x$. 
\end{theorem}

Let us recall a contraction property of gradient flows. 
\begin{lemma}
Let $c_1, c_2$ be two gradient curves of a $\lambda$-concave function $f$ defined on $U$. 
Suppose that $c_1(t)$ and $c_2(t)$ can be joined by a minimal geodesic contained in $U$, for every $t$ with $t_1 \ge t \ge t_0$. 
Then, we have $|c_1(t)c_2(t)| \le e^{\lambda (t-t_0)} |c_1(t_0) c_2(t_0)|$ for $t_1 \ge t \ge t_0$.
\end{lemma}
\begin{proof}
We may assume that $t_0 = 0$ and set $x_1 = c_1(0)$ and $x_2 = c_2(0)$. 
Let $\gamma : [0,|x_1x_2|] \to U$ be a geodesic with $\gamma(0)=x_1$ and $\gamma(|x_1x_2|)=x_2$.
Since $f$ is $\lambda$-concave along $\gamma$, we have 
\[
\frac{f(x_i)-f(x_j)-\frac{\lambda}{2}|x_1x_2|^2}{|x_1x_2|} \le f'(\uparrow_{x_j}^{y}) \le \left< \nabla f, \uparrow_{x_j}^{y} \right>
\]
for $(i,j)=(1,2), (2,1)$, where $y$ is the midpoint in $\gamma$.

On the other hands, we have 
\begin{align*}
\left. \frac{d}{d t} \right|_{t=0+} |c_1(t)c_2(t)| 
&\le (d_y)_{x_1}'(c_1^+(0)) + (d_y)_{x_2}'(c_2^+(0)) \\
&\le - \left< \uparrow_{x_1}^y, \nabla f \right> - \left< \uparrow_{x_2}^y, \nabla f \right> \\
&\le \lambda |x_1x_2|.
\end{align*}
This immediately implies the assertion. 
\end{proof}

Let us recall the definition of polar vectors. 
\begin{definition}[\cite{Pet}] \upshape
Let $C$ be a Euclidean cone of nonnegative curvature. 
For a vector $u \in C$ and a subset $\mathcal V \subset C$, we say that $u$ is {\it polar to} $\mathcal V$ if 
\[
\left<u, w \right> + \sup_{v \in \mathcal V} \left<v, w\right> \ge 0
\]
for any $w \in C$.
\end{definition}
Note that $u$ is polar to $\mathcal V$ if and only if 
\begin{equation} \label{eq:polar}
\phi(u) + \inf_{v \in \mathcal V} \phi(v) \le 0
\end{equation}
holds for any concave function $\phi : C \to \mathbb R$.
A geometric meaning of vector being polar is explained as follows. 
For vectors $v,w \in C$, if $|v|=|w|$, then $v$ is polar to $w$ if and only if $\angle v o z + \angle w o z \le \pi$ for any $z \in C$. 
So, if the space of directions at the origin of $C$ has diameter not greater than $\pi/2$, any two vectors of the same length are polar.
If $C$ isometrically splits as $C = C' \times \mathbb R$, then $(o,t)$ is polar to $(o,-t)$, where $o \in C'$ is the origin and $t > 0$.

\begin{lemma}[\cite{Pet}] \upshape \label{lem:polar}
For a point $p$ in an Alexandrov space $M$ and a closed subset $A$ of $M$ with $p \not\in A$, the gradient $\nabla_p\, d_A$ is polar to $A_p'$. 
Here, $A_p'$ is the set of all directions of a minimal geodesic from $p$ to $A$. 
\end{lemma}
\begin{proof}
Let us fix $w \in T_pM$. 
Let $\xi \in A_p'$ be a direction such that $(d_A)'(w) = - \max_{\eta \in A_p'} \left< \eta, w \right> = - \left< \xi, w \right>$. 
Then, we have 
\[
\left< \nabla d_A, w \right> + \left< \xi, w \right> \ge 0.
\]
This completes the proof.
\end{proof}

%
%

\section{Strongly Lipschitz contractible balls} \label{sec:SLC}

In this section, we prove 
\begin{theorem} \label{prop:SLCB}
Let $p$ be a point in an Alexandrov space and $r > 0$.
If $d_p$ is regular on $B(p,r) \setminus \{p\}$, then $B(p,r)$ is strongly Lipschitz contractible to $p$. 
\end{theorem}
This is a global version of Theorem \ref{thm:SLCB}. 
To prove Theorem \ref{prop:SLCB}, we prove 

\begin{theorem} \label{thm:SLC}
Let $f$ be a proper semiconcave function defined on an open set $U$ such that $f$ is regular on $f^{-1}[a,b]$ for some $a < b$.
Then, there is a Lipschitz map $H : \{f \le b \} \times [0,1] \to \{f \le b \}$ 
such that for every $x \in f^{-1}[a,b]$, $y \in \{f \le a\}$ and $t \in [0,1]$, we have 
\begin{itemize}
\item $H_0(x) = x$, $f(H_1(x))=a$; 
\item the function $f(H_t(x))$ is monotone non-increasing in $t$; 
\item $H_t(y)=y$.
\end{itemize}
\end{theorem}
This theorem is proved in \S \ref{sec:proof}.
Remark that for an $f$ as in Theorem \ref{thm:SLC}, the gradient flow of $f$ increases the value of $f$. 
Since Theorem \ref{thm:SLC} gives a ``reverse flow'' of it in some sense, the existence of such a flow is non-trivial.
Such a reverse flow is important for applications.


\begin{proposition} \label{prop:regular}
Let $f$ be a semiconcave proper function defined on an open subset $U$ in an Alexandrov space $M$. 
Suppose that there is an $r \in \mathbb R$ such that 
$f$ is regular on $f^{-1}(r)$. 
Then, there exist $r',r'',\bar r \in f(U)$ with $r' < r < r'' < \bar r$ such that the distance function $d_{f^{-1}(\bar r)}$ from the level set $f^{-1}(\bar r)$ is regular on $f^{-1}[r',r''] \subset U$. 
Further, 
\[
f'(\nabla d_{f^{-1}(\bar r)}) < - c
\]
holds on $f^{-1}[r',r'']$, for some $c > 0$.
\end{proposition}

\begin{proof}
Since $f^{-1}(r)$ is compact and $f$ is regular on $f^{-1}(r)$, by the lower semicontinuity of the absolute gradient, $|\nabla f| > c$ on $f^{-1}(r)$ for some $c > 0$.
Let $\lambda$ be such that $f$ is $\lambda$-concave near $f^{-1}(r)$.
We may assume that $\lambda \ge 0$. 
Let $\nu > 0$ be taken so that 
for any $x \in f^{-1}(r)$ and $y \in U$ with $|xy| < \nu$, every minimal geodesic between them is contained in $U$. 
For instance, we set $\nu$ the half of $|f^{-1}(r), M \setminus U|$. 

First, we prove that there are $\delta > 0$ and $\bar \ell > 0$ such that 
for any $x \in f^{-1}[r-\delta, r+ \delta]$, there is $y \in U$ with $\nu > |xy| > \bar \ell$ and 
\[
\frac{f(y)-f(x)-\frac{\lambda}{2}|xy|^2}{|xy|} > c.
\]
Using it, we completes the proof of the lemma.

By the assumption, for any $x \in f^{-1}(r)$, there exists $y \in U$ with $|yx| < \nu$ 
such that 
\[
\frac{f(y)-f(x)-\frac{\lambda}{2}|yx|^2}{|yx|} > c. 
\]
Fixing $x$ and $y$, there is $\e > 0$ such that if $z \in B(x,\e)$, then $|zy|< \nu$ and 
\[
\frac{f(y)-f(z)-\frac{\lambda}{2}|yz|^2}{|yz|} > c.
\]

Since $f^{-1}(r)$ is compact, there are finitely many points $x_1, \dots, x_m \in f^{-1}(r)$, $y_1, \dots, y_m \in U$ 
and positive numbers $\e_1, \dots, \e_m$ such that $f^{-1}(r) \subset \bigcup_{1 \le i \le m} B(x_i,\e_i) \subset U$ and that if $x \in B(x_i,\e_i)$, then $|xy_i| < \nu$ and 
\[
\frac{f(y_i)-f(x)-\frac{\lambda}{2}|y_ix|^2}{|y_ix|} > c.
\]
There is $\delta_0 > 0$ such that $f^{-1}[r-\delta_0,r+\delta_0] \subset \bigcup_{1 \le i \le m} B(x_i,\e_i)$. 

Now, for each $x \in f^{-1}[r-\delta,r+\delta]$, let us define the value $\ell(x)$ as follows. 
Setting $L_x = \left\{y \in U \relmiddle| |xy| < \nu \text{ and } \frac{f(y)-f(x)-\frac{\lambda}{2}|xy|^2}{|xy|} > c \right\}$ and $\ell(x) = \sup \{|xy| \mid y \in L_x \}$. 
Obviously, the function $x \mapsto \ell(x)$ is lower semicontinuous.
We set $\ell_\delta := \min\{\ell(x) \mid x \in f^{-1}[r-\delta,r+\delta]\}$ for $0 \le \delta \le \delta_0$. 
Then, $\ell_\delta$ converges to $\ell_0$ as $\delta \to 0$. 
Since $\ell_0 > 0$, some $\delta > 0$ exists so that 
\[
r - \delta + c \ell_\delta > r + \delta.
\] 
We fix some constant $\bar r$ with $r - \delta + c \ell_\delta > \bar r > r + \delta$, and define $\bar \ell > 0$ by $\bar r = r - \delta + c \bar \ell$.

Then, for any $x \in f^{-1}[r-\delta,r+\delta]$, there is $y \in U$ with $\bar \ell < |xy| < \nu$ and 
\[
\frac{f(y)-f(x)-\frac{\lambda}{2}|yx|^2}{|yx|} > c. 
\]
Note that $f(y) > f(x) + c |xy| > \bar r > r +\delta \ge f(x)$.
So, there is a point $z$ in a geodesic between $x$ and $y$ such that $f(z) = \bar r$. 
By the $\lambda$-concavity of $f$, we obtain 
\[
\frac{f(z)-f(x)-\frac{\lambda}{2}|zx|^2}{|zx|} \ge \frac{f(y)-f(x)-\frac{\lambda}{2}|yx|^2}{|yx|} > c.
\]
Let $w \in f^{-1}(\bar r)$ be a point so that $|xw| = \min\{|x \bar w| \mid \bar w \in f^{-1}(\bar r) \}$. 
Then, we have 
\[
f_x'(\uparrow_x^w)+\frac{\lambda}{2}|wx| \ge \frac{f(w)-f(x)}{|wx|} \ge \frac{f(z)-f(x)}{|zx|} > c+\frac{\lambda}{2}|zx|.
\]
Hence, $f_x'(\uparrow_x^w) > c$. 
{By Lemma \ref{lem:polar},} $\nabla_x d_{f^{-1}(\bar r)}$ is polar to $f^{-1}(\bar r)_x' \subset \Sigma_x$. 
{Hence, by} \eqref{eq:polar}, we obtain
\[
f_x'(\nabla d_{f^{-1}(\bar r)}) 
< - c.
\]
This completes the proof.
\end{proof}

Proposition \ref{prop:regular} enables us to check that the gradient flow of the distance function from $f^{-1}(\bar r)$ makes a Lipschitz flow whose flow curves decrease the value of $f$. 
When the curves arrive at the level set $f^{-1}(r')$, we use Proposition \ref{prop:regular} again and obtain the gradient flow of the distance function from some level set $f^{-1}(r'+\epsilon)$ for some $\epsilon > 0$. 
Then, we connect two flows on $f^{-1}(r')$ and that check that the obtained flow is also Lipschitz, in the next two subsections.

\subsection{Lipschitz regularity of an implicit function}
Let $f$ be a proper semiconcave function defined on an open set $U$ which is regular on $U$. 
Let $\Phi$ denote the maximal gradient flow of $f$. 
For $x \in U$, the maximal time defining the flow $\Phi(x,\,\cdot)$ on $U$ is denoted by $T_x$.
We assume that there are a proper semiconcave function $g$ defined on $U$, real numbers $a < b$ and $c > 0$ such that $g(U) \supset [a,b]$ and that 
\begin{equation} \label{eq:decreasing}
g_x'(\nabla_x f) < -c
\end{equation}
for every $x$ in a neighborhood of $g^{-1}[a,b]$. 
Further, we assume that for some $\bar a < a$ and $b < \bar b$, we may assume that $g_x'(\nabla_x f) < -c$ on $g^{-1}[\bar a, \bar b]$. 
In particular, $g(\Phi(x,t))$ is strictly decreasing in $t$ whenever $\Phi(x,t) \in g^{-1}[\bar a, \bar b]$. 
For any $x \in g^{-1}[a,b]$, we define the first hitting time to $\{g \le a\}$ of $x$ by 
\[
t(x) := \min \{t \in [0,T_x) \mid \Phi(x,t) \in \{g \le a\} \}. 
\]
The condition \eqref{eq:decreasing} implies that the set of all $t$'s with $\Phi(x,t) \in \{g\le a\}$ has the form $[t(x), T_x)$. 
Further, $g(\Phi(x,t))=a$ if and only if $t = t(x)$. 
Then, $x \mapsto t(x)$ can be checked to be continuous. 
We also easily check that some $T$ exists so that $t(x) \le T$ for all $x \in g^{-1}[a,b]$.
For instance, we set $T = (b-a)/c$. 

If $\epsilon > 0$ is taken to be so small, then we have that for any $x,y \in g^{-1}[a,b]$ with $|xy| < \e$, every minimal geodesic segment between $x$ and $y$ is contained in $g^{-1}[\bar a, \bar b]$. 
Indeed, we take $\epsilon$ as a positive number smaller than $\min\{a-\bar a, \bar b - b\} / 2 \Lip(g)$.

\begin{lemma}[Implicit function lemma] \label{lem:implicit function}
Let $f,g,U,a,b,\bar a,\bar b, \e$ be as above. 
Then, the function $g^{-1}[a, b] \ni x \mapsto t(x) \in [0,T]$ is Lipschitz continuous. 
Further, if $x,y \in g^{-1}[a, b]$ with $|xy| < \e$, then we have 
\[
|t(x) - t(y)| \le L(f,g,c,a,b,\e)|xy|.
\]
for some constant $L(f,g,c,a,b,\e)$ depending on $f,g,c,a,b,\e$.
\end{lemma}
\begin{proof}
For $x \in g^{-1}[a,b]$, if $\Phi(x,t) \in g^{-1}[\bar a, \bar b]$, then we have 
\[
\left(\frac{d}{d t} \right)_{\hspace{-4pt}+} g(\Phi(x,t)) = g_{\Phi(x,t)}'(\nabla f) < - c.
\]
Let $\lambda$ be a constant so that $f$ is $\lambda$-concave on $g^{-1}[a,b]$. 
Then, $\Lip(\Phi(\cdot,t)) \le e^{\lambda t}$ on $U$. 
If a geodesic segment $\gamma$ of constant speed is contained in $g^{-1}[\bar a, \bar b]$, then the function $g(\Phi(\gamma(s),t))$ is Lipschitz in $s$, so it has the derivative for almost all $s$ with 
\begin{align*}
\left|\frac{d}{d s} g(\Phi(\gamma(s),t)) \right| \le \Lip(g) e^{\lambda t} |\dot \gamma(s)|.
\end{align*}
Let us take points $x,y \in g^{-1}[a,b]$ with $|xy|< \epsilon$ and a geodesic segment $\gamma : [0,1] \to g^{-1}[\bar a, \bar b]$ of constant speed $|\dot\gamma(s)| \equiv |xy|$ with $\gamma(0) = x$ and $\gamma(1) = y$. 
We assume that $t(y) > t(x)$. 
Then, setting $\sigma(s) = t(x) + s(t(y)-t(x))$, we have 
\begin{align*}
0 &= g(\Phi(y,t(y))) - g(\Phi(x,t(x))) \\
&= \int_0^1 \frac{d}{d s} g(\Phi(\gamma(s),\sigma(s))) \, ds \\
&= \int_0^1 \left. \frac{d}{d s_1} \right|_{s_1=s} \hspace{-10pt} g(\Phi(\gamma(s_1),\sigma(s))) + \left. \frac{d}{d s_2} \right|_{s_2=s} \hspace{-10pt} g(\Phi(\gamma(s),\sigma(s_2))) \, ds
\end{align*}
Now, we have 
\begin{align*}
\left. \frac{d}{d s_2} \right|_{s_2=s} \hspace{-10pt} g(\Phi(\gamma(s),\sigma(s_2))) 
&= g'(\nabla f) \dot \sigma(s) 
= g'(\nabla f) (t(y)-t(x)) \\
&< - c(t(y)-t(x))
\end{align*}
Hence, we obtain 
\begin{align*}
c (t(y)-t(x)) &< \int_0^1 \left. \frac{d}{d s_1} \right|_{s_1=s} \hspace{-10pt} g(\Phi(\gamma(s_1),\sigma(s))) \, ds \\
&\le \int_0^1 \left| \left. \frac{d}{d s_1} \right|_{s_1=s} \hspace{-10pt} g(\Phi(\gamma(s_1),\sigma(s))) \right| \, ds \\
&\le \Lip(g) e^{\lambda T} |\dot \gamma(s)| = \Lip(g) e^{\lambda T} |xy|
\end{align*}
This provides the second assertion in the conclusion. 
Since $g^{-1}[a,b]$ is compact, $t(\cdot)$ is Lipschitz on $g^{-1}[a,b]$.
This completes the proof. 
\end{proof}

\subsection{Gluing two gradient flows}

Let $U$ be a bounded open subset of an Alexandrov space, and $g,h$ semiconcave {functions} defined on $U$. 
Let $f$ be a semiconcave function defined on an open set $V$ with $V \subset U$. 
Suppose that there exist $a,b,c,d$ with $a < b < c < d$ so that $\{h \le c\} \subset V$, $\{h \le d\} \subset U$ and that 
\begin{align*}
h'(\nabla f) &< -A \text{ on a neighborhood of } h^{-1}[a,c], \\
h'(\nabla g) &< -A \text{ on a neighborhood of } h^{-1}[b,d]
\end{align*}
for some constant $A > 0$. 
Let $\Phi$ and $\Psi$ denote the gradient flows of $f$ and $g$, respectively. 

\begin{lemma}[Gluing lemma]\label{lem:gluing}
Let $U,V,f,g,h,a,b,c,d,\Phi,\Psi$ be as above. 
Then, there exists a locally Lipschitz map $H : \{h \le d\} \times [0,\infty) \to \{h \le d\}$ such that 
\begin{itemize}
\item $H(x,0)=x$ for $x \in \{h \le d\}$; 
\item $H(x,t)= \Psi(x,t)$ for $(x,t) \in h^{-1}[c,d] \times [0,\e]$; 
\item $H(x,t)=\Phi(x, t - T)$ for $x \in \{h \le b\}$ and $t \in [T,T_x)$, 
\end{itemize}
for some $T > 0$ and $\e > 0$.
Further, the function $h(H(x,t))$ is monotone non-increasing in $t$ for every $x \in \{h \le d\}$. 
\end{lemma}
\begin{proof}
For any $y \in h^{-1}[b,d]$, we set $t(y) = \min\{t \in [0,\infty) \mid \Psi(y,t) \in \{h \le b\}\}$. 
By Lemma \ref{lem:implicit function}, the function $h^{-1}[b,d] \ni y \mapsto t(y)$ is Lipschitz. 
Let $T = \max\{t(y) \mid y \in h^{-1}[b,d]\} \le (d-b)/A$. 
Let us define the map $H : \{h \le d\} \times [0,\infty) \to \{h \le d\}$ by 
\[
H(x,t) = \left\{
\begin{aligned}
&\Psi(x, t) &&\text{if } x \in h^{-1}[b,d], t \le t(x) \\
&\Psi(x,t(x)) &&\text{if } x \in h^{-1}[b,d], t \in [t(x),T] \\
&\Phi(\Psi(x,t(x)),t-T) &&\text{if } x \in h^{-1}[b,d], t \ge T \\
&x &&\text{if } x \in \{h \le b\}, t \in [0,T] \\
&\Phi(x,t-T) &&\text{if } x \in \{h \le b\}, t \ge T. 
\end{aligned}
\right.
\]
Let us set $\e = \min \{t(y) \mid y \in h^{-1}[c,d]\} > 0$. 
The map $H$ satisfies that $H(y,t) = \Psi(y,t)$ for $y \in h^{-1}[c,d]$ and $t \le \e$. 
\end{proof}

\subsection{Proof of Theorems \ref{prop:SLCB} and \ref{thm:SLC}} \label{sec:proof}
Let us first prove Theorem \ref{thm:SLC}.

\begin{proof}[Proof of Theorem \ref{thm:SLC}]
Let $f : U \to \mathbb R$ be a proper semiconcave function which is regular on $f^{-1}[a,b]$ for some $a < b$. 
Let $c > 0$ be a number satisfying $|\nabla f| > c$ on $f^{-1}[a,b]$.
By Proposition \ref{prop:regular}, there are $a = t_0 < t_1 < \dots < t_N = b$ with finite sequences of positive numbers $\{\tau_i\}_{i=1}^N$ and $\{\delta_i\}_{i=1}^N$ such that $\delta_i < \tau_i$ and 
\[
f'(\nabla_x d_{f^{-1}(t_i+\tau_i)}) < - c
\]
on $f^{-1}[t_i-\delta_i, t_i+ \delta_i]$, and that $\bigcup_{i=1}^N (t_i-\delta_i,t_i+\delta_i) \supset [a,b]$. 
Using Lemma \ref{lem:gluing} repeatedly, we obtain a Lipschitz map 
\[
H : \{f \le b\} \times [0,T] \to \{f \le b\}
\]
such that $H_0(x)=x$, $H_T(x) \in \{f = a\}$ and $f(H_t(x))$ is monotone non-increasing in $t$ for every $x \in \{f \le b\}$.
Further, $H(x, t)$ coincides with the gradient flow $\Phi(x, t-A)$ of $d_{f^{-1}(t_0+\tau_0)}$ with some parameter translation $A$, if $t$ is close to $T$ and $f(x)$ is close to $a$.
For $x \in \{f \le b\}$, we set $t(x) := \min \{t \ge 0 \mid H(x,t) \in \{f \le a\}\}$. 
Then, by Lemma \ref{lem:implicit function}, the map $t(\cdot)$ is Lipschitz.
Let us define $G : \{f \le b\} \times [0,T] \to \{f \le b\}$ by 
\[
G(x,t) := \left\{ 
\begin{aligned}
& H(x,t) &&\text{ if } t \le t(x) \\
& H(x,t(x)) &&\text{ if } t \ge t(x).
\end{aligned}
\right.
\]
Then, $G$ is Lipschitz such that $G(x,0) = x$ and $G(x,T) \in \{f = a\}$ for all $x \in \{f \le b\}$ and $G(y,t) = y$ for all $y \in \{f \le a\}$ and $t \in [0,T]$. 
Further, $f(G(x,t))$ is monotone non-increasing in $t$, for every $x \in f^{-1}[a,b]$. 
This completes the proof. 
\end{proof}

\begin{remark} \label{rem:SLC} \upshape
Remark that if a semiconcave function $f$ is globally defined on a compact Alexandrov space $X$ and has the following gradient estimate 
\[
(d_p)'(\nabla_x f) < - c
\]
for every $x \in X \setminus \{p\}$, for some $p \in M$ and a uniform constant $c > 0$, then the gradient flow $\Phi$ of $f$ can reach $p$ in a uniform finite time. 
Hence, up to time scaling, $\Phi$ gives a strong Lipschitz contraction from $X$ to $p$.

If no such a gradient estimate of a semiconcave function exists, then its gradient flow may not give a strong Lipschitz contraction. 
Indeed, there is a strictly concave function such that the gradient flow does not reach its unique critical point in any finite time. 

Let us consider a strictly concave function $f(x) = -x^2 /2$ on $[-1,1]$. 
The unique critical point is the zero $0$. 
For $x \in (0,1]$, the gradient of $f$ at $x$ is determined as follows. 
\[
|\nabla f|_x = x \text{ and } \nabla f(x) = x \uparrow_x^0.
\]
Let us consider the gradient curve $\alpha : [0,\infty) \to [-1,1]$ of $f$ with $\alpha(0) = 1$. 
In this case, it is determined by the following single differential equation. 
\[
|\alpha^+(t)| = |\nabla f|_{\alpha(t)}.
\]
Since 
\[
|\alpha^+(t)| = \lim_{\delta \to 0+} \frac{|\alpha(t+\delta)-\alpha(t)|}{\delta} = \lim_{\delta \to 0+} \frac{\alpha(t)-\alpha(t+\delta)}{\delta}, 
\]
we have 
\[
\alpha(t) = e^{-t}.
\]
Therefore, the curve $\alpha(t)$ does not reach $0$ in finite time. 
\end{remark}

\begin{proof}[Proof of Theorem \ref{prop:SLCB}]
Let $p$ be a point in an Alexandrov space $M$ and $r > 0$ such that $d_p$ is regular on $B(p,r) \setminus \{p\}$. 
By Theorem \ref{thm:SLCB}, there is $r_0 > 0$ and a strong Lipschitz contraction $F$ from $B(p,r_0)$ to $p$. 
So, we may assume that $r_0 < r$. 
Then, by the assumption, $d_p$ is regular on $d_p^{-1}[r_0,r]$. 
From Theorem \ref{thm:SLC}, there is a Lipschitz map 
\[
H : B(p,r) \times [0,1] \to B(p,r)
\]
such that $H(x,0)=x$, $H(x,1) \in S(p,r_0)$ and $d_p(H(x,t))$ is monotone non-increasing in $t$ for every $x \in B(p,r_0)$ and that $H(y,t) = y$ for every $(y,t) \in B(p,r_0) \times [0,1]$. 
Gluing two homotopies $F$ and $H$ in a natural way, we obtain a strong Lipschitz contraction from $B(p,r)$ to $p$. 
This completes the proof. 
\end{proof}

\section{Existence of good covering and homotopy types} \label{sec:good-cov}

Let $M$ be an Alexandrov space. 
A subset of $M$ is called a {\it domain} if it is a connected open subset. 
A domain $U$ of $M$ is {\it conical} if there are $x \in U$ and a topological space $A$ such that $(U,x)$ is homeomorphic to the open cone $(K(A), o)$ as a pointed space.
Then, $U$ is called a conical neighborhood of $x$ and $A$ is called a generator of $U$. 
From the uniqueness of conical neighborhoods (\cite{Kw}) and Perelman's stability theorem
 (\cite{Per Alex II}, \cite{Kap stab}),
$(U,x)$ must be homeomorphic to $(K(\Sigma_x),o)$. 
So, the generator $A$ is compact and has the same (co)homology groups as
those of $\Sigma_x$. 
However note that $A$ is not homotopic to $\Sigma_x$, in general.
For instance, if $M$ is the cone over the suspension of a homology
sphere $X$, 
then the apex $o$ of $M$ has a conical neighborhood with generator
homeomorphic to a sphere, 
however $\Sigma_o = \Sigma(X)$ is not homeomorphic to a sphere.

A subset $V$ of $M$ is called {\it convex} if every minimal geodesic
segment joining any two points of $V$ is contained in $V$. 
Here it should be noted that the uniqueness of geodesics does not hold in
general as shown by the double of a flat disk.

To prove Theorem \ref{thm:good-cover}, we need 

\begin{theorem} [cf. \cite{Per}, \cite{Kap}, \cite{MY:stab}]\label{thm:concave function}
For any $p$ in an Alexandrov space $M$, there exist an open neighborhood $\Omega$ of $p$ 
and a strictly concave function $f$ defined on $\Omega$ such that
\begin{enumerate}
\item $f(p)=\max_{\Omega} f;$ 
\item $\{ f > c \}$ is convex and conical SLC to $p$ for any $c$ with $\inf_{\Omega} f < c\le \max_{\Omega} f$.
\end{enumerate}
\end{theorem}

First it should be noted that the strictly concave function $f$ in Theorem \ref{thm:concave function} 
was constructed in \cite{Per}, \cite{Kap} as the minimum of the average of the composition of 
distance functions and a strictly concave $C^2$-function, 
by using some net in a metric sphere around $p$. 
More explicitly this is done as follows : Let $r>0$ be small enough.
Fixing some maximal $\ell$-discrete set $\{x_\alpha \}_{\alpha}$ of $S(p, 2r)$, a maximal $\nu$-discrete set $\{x_{\alpha\beta}\}_{\beta}$ of 
$S(p, 2r) \cap B(x_{\alpha}, 2 \ell r)$ with $\nu \ll \ell$ and a concave increasing function 
$\chi : (0, 3r) \to \mathbb R$ 
which is strictly concave near $r$, we set $f_{\alpha} = \frac{1}{\# \{\beta\}} \sum_{\beta} \chi(d(x_{\alpha\beta}, \cdot))$ and 
$f = \min_\alpha f_{\alpha}$. 
Then $f$ is strictly concave and regular on $U(p,r)$ except $p$,
{where $U(p,r)$ denotes the open $r$-ball around $p$}.
 It is checked that the set $\{ f >c\}$ is 
SLC to $p$ for some $c$ with $c< \max _{U(p,r)} f$ (see \cite{MY:stab}).
We only have {to} show that $\{ f>c\}$ is conical. If $p\notin \partial M$, Theorem \ref{thm:fibration} implies 
the conclusion. If $p\in\partial M$, we take the metric ball $\tilde B(p,3r)$ in the double $D(M)$,
and take an $\ell$-discrete set $\{ \tilde x_\alpha \}_{\alpha}$ of $\tilde S(p, 2r) :=\partial \tilde B(p, 2r)$, 
a maximal $\nu$-discrete set $\{\tilde x_{\alpha\beta}\}_{\beta}$ of 
$\tilde S(p, 2r) \cap \tilde B(x_{\alpha}, 2 \ell r)$ in such a way that those are invariant under the action 
of reflection with respect to $\partial M$.
Then the function $\tilde f:\tilde U(p,r)\to \mathbb R$ defined by the distance functions from those points
in a similar way to the above is strictly concave and regular except $p$.
Thus by Theorem \ref{thm:fibration} $f$ is a fiber bundle when restricted to $\{ {\max} f> f> c\}$, and hence
$\{ f>c\}$ must be conical. 

\begin{proof}[Proof of Theorem \ref{thm:CSLC}]
Take $a$ with $\inf_U f < a <\max_U f$, and let $\Omega:=\{ f >a\}$.
Since $f$ is strictly concave, a maximizer of it is unique, say $p \in \Omega$.
Then, for any $x \in \Omega \setminus \{p \}$, we have from the concavity of $f$
\begin{equation} \label{eq:grad}
f'(\uparrow_x^p) \ge \frac{f(p)-f(x)}{|px|} > 0. 
\end{equation}
Therefore $f$ is regular on $\Omega \setminus \{p \}$. 
Let us take $c > 0$ such that $f$ is $(-c)$-concave on $\Omega$. 
For $x \neq p$, the $(-c)$-concavity implies 
\[
f(x) \le f(p) + f'(\uparrow_p^x) |px| - (c/2) |px|^2 \le f(p) - (c/2) |px|^2.
\]
Therefore, \eqref{eq:grad} is improved by 
\begin{equation} \label{eq:grad2}
f'(\uparrow_x^p) \ge \frac{f(p)-f(x)}{|px|} \ge (c/2)|px|.
\end{equation}
Hence, $f'(\uparrow_x^p)$ has a uniform lower bound on $\{|px| \ge r\}$ depending on $r$, for every fixed $r > 0$. 
By the first variation formula, we have 
\[
(d_p)'(\nabla_x f) \le - \left<\uparrow_x^p, \nabla_x f \right> \le - f'(\uparrow_x^p).
\]
This together with \eqref{eq:grad2} implies 
\[
(d_p)'(\nabla_x f) \le - (c/2)r.
\]
for every $x \in \{d_p \ge r\}$.
Take $r>0$ with $B(x,r)\subset\Omega$. From Theorem \ref{thm:SLC}, there is a Lipschitz homotopy
\[
F : \Omega \times [0,1] \to \Omega
\]
such that $F(x,0)=x$, $|p, F(x,1)|=r$ and $|p, F(x,t)|$ is monotone non-increasing in $t$ for every $x \in \{d_p \ge r\}$ and that $F(y,t)=y$ for every $(y,t) \in B(p,r) \times [0,1]$. 

On the other hands, Theorem \ref{thm:SLCB} gives a strong Lipschitz contraction $G$ from $B(p,r)$ to $p$, if $r$ is small. 
Gluing two Lipschitz homotopies $F$ and $G$ in a natural way, we have a Lipschitz homotopy 
\[
H : \Omega \times [0,1] \to \Omega
\]
such that $H(x,0)=x$, $H(x,1)=p$ and $|p, H(x,t)|$ is monotone nonincreasing in $t$, for every $x \in \Omega$. 
This completes the proof of $(1)$. 

For the proof of $(2)$, we only have to use Theorem \ref{thm:fibration}. 
\end{proof}

\begin{lemma} \label{lem:SLC}
Let $U_1,\ldots,U_m$ be convex, conical SLC domains in $M$ 
defined as superlevel sets $U_i=\{f_i>c_i\}$ via strictly concave functions $f_i$ as in 
Theorem \ref{thm:concave function} defined on domains $\Omega_i$. 
If $U_1\cap \cdots \cap U_m$ is nonempty, it is a convex, conical SLC-domain.
\end{lemma}

\begin{proof}
We may assume $U_i=\{ f_i > 0\}$, where $f_i$ is $(-c)$-concave for some $c>0$.
Set $\Omega:= \bigcap_{i=1}^m \Omega_{i}$
Then 
\[
U_{1} \cap \cdots \cap U_{m} = \{x \in \Omega \mid \min_{1 \le i \le m} f_{i}(x) > 0 \}.
\]
Since $\min_{1 \le i \le m} f_{i}$ is $(-c)$-concave on $\Omega$, the conclusion follows from Theorem \ref{thm:CSLC} if $\Omega$ does not meet $\partial M$.
In case $\Omega$ meets $\partial M$, we first construct $\Bbb Z_2$-equivariant $(-c)$-concave function 
 $\tilde f_i$ on the double $D(U_i)$ in a way similar to the construction right after Theorem \ref{thm:concave function}.
Therefore $\tilde f_i$ descends to a $(-c)$-concave function $f_i$ on $U_i$, and 
 again we can apply Theorem \ref{thm:CSLC} to get that
the set $\bigcap_{1 \le i \le m} U_{i}$ is a convex, conical SLC domain.
\end{proof}

\begin{proof}[Proof of Theorem \ref{thm:good-cover}(1)]
Let $\mathcal V$ be an open covering of $M$. 
For any $x \in M$, we fix $V_x \in \mathcal V$ with $x \in V_x$. 
By Theorem \ref{thm:concave function}, there is a strictly concave
function 
$f_x$ defined on some neighborhood $\Omega_x$ of $x$ with $\Omega_x \subset V_x$. 
Adding a constant to $f_x$, we may assume $U_x = \{ y\in \Omega_x \mid f_x(y) > 0\}$. 
By Lemma \ref{lem:SLC}, $U_x$ is a conical, convex SLC neighborhood of $x$.
Since $M$ is proper, it is covered by a countable union of compact
subsets.
Therefore we can choose a countable set $\{ x_i\}\subset M$ such that
$\{ U_{x_i}\}$ is a locally finite covering of $M$. 
If the intersection $U_{x_1} \cap \dots \cap U_{x_m}$ is nonempty, we
can set 
\[
U_{x_1} \cap \dots \cap U_{x_m} = \{x \in \bigcap_{i=1}^m \Omega_{x_i} \mid \min_{1 \le i \le m} f_{x_i}(x) > 0 \},
\]
it must be a convex, conical SLC domain by Lemma \ref{lem:SLC}.
This completes the proof.
\end{proof}

\begin{proof}[Proof of Theorem \ref{thm:good-cover}(2)]
We consider only the case that $M$ is noncompact.
Let $\mathcal U = \{ U_i \}_{i=1}^{\infty}$ be a locally finite good covering of $M$.
Recall $U_i$ is defined as the super level set 
\[
U_i = \{ \varphi_i > 0 \},
\]
of a strictly concave function $\varphi_i$. Define $\psi_i = \varphi_i -c_i$ 
on $U_i$, and $\psi_i= 0$ outside $U_i$, 
and set 
\[
f_i = \frac{\psi_i}{\sum_{j=1}^{\infty} \psi_j}.
\]

Note that $f_i$ is a Lipschitz function on $M$ satisfying 
\begin{enumerate}
\item $0\le f_i\le 1;$
\item $f_i > 0$ on $U_i$ and $\supp(f_i) = \bar U_i;$
\item $\sum_{i=1}^{\infty} f_i = 1$.
\end{enumerate}
Let $K$ be the nerve of the covering $\{ U_i\}_{i=1}^{\infty}$ where
the vertices of $K$ are the canonical basis $\{ e_i\}_{i=1}^{\infty}$ of
${\mathbb R}^{\infty}$. Define $F:M \to |K|$ by 
\[
F(x)= (f_1(x),\ldots, f_N(x), \ldots).
\]
Let $K'$ be the barycentric subdivision of $K$. 
Recall that 
$U_{i_0\cdots i_m} := U_{i_0}\cap \cdots\cap U_{i_m} = \{ \varphi_{i_0\cdots i_m} > 0 \}$, 
where $\varphi_{i_0\cdots i_m} := \min_{j=0}^m \varphi_{i_j}$, and that
$U_{i_0}\cap \cdots\cap U_{i_m}$ is SLC to 
the unique maimum point, denoted by $p_{i_0\cdots i_m}$, of $\psi_{i_0\cdots i_m}$
via the gradient curves of $\psi_{i_0\cdots i_m}$. 
We now define $G:|K'| \to M$ as follows.
For any $\sigma=[e_{i_0},\ldots, e_{i_k}]\in K$, let $b(\sigma)$ be the
barycenter of $\sigma$. We put $G(b(\sigma)) := p_{i_0\cdots i_k}$. 
Assume that $G$ is defined on the $(m-1)$-skelton $|(K')^{m-1}|$ of 
$K'$ in such a way that if all the vertices of an $(m-1)$-simplex $\tau$ of
$K'$ is mapped via $G$ to $U_{k_0\cdots k_{\ell}}$, $G(\tau)$ is also contained 
in $U_{k_0\cdots k_{\ell}}$. 
Now take any $m$-simplex $s=[b_{0}\cdots b_{m}]$ of $K'$. 
Let $U_{j_0},\ldots, U_{j_{\ell}}$ be the set containing all of $G(b_0),\ldots, G(b_m)$.
By the inductive assumption $G(\partial s)$ is also contained in 
$U_{j_0\cdots j_{\ell}}$.
Lemma\ref{lem:SLC} enables us to extend $G:\partial s \to U_{j_0\cdots j_{\ell}}$
to a Lipschitz map $G:s \to U_{j_0\cdots j_{\ell}}$ by deforming $G(\partial s)$ to $p_{j_0\cdots j_{\ell}}$.
Repeating this procedure, we have a Lipschitz map $G:|K'| \to M$.

\begin{assertion}\label{ass:FG}
$F\circ G$ is Lipschitz homotopic to the identity $1_K$. 
\end{assertion}

\begin{proof}
For any $x=(x_1,\ldots, x_N, \ldots)\in |K'|$, let $s$ and $\sigma$ be the open simplexes 
of $K'$ and $K$ respectively containing $x$.
Let $\sigma = |e_{i_0}\cdots e_{i_{\ell}}|$.
Take $0\le j\le \ell$ with $G(x)\in U_{i_j}$.
It follows that $f_{i_j}(G(x))>0$.
Note $x_{i_j}>0$. Set $h_i(x)= \min\{ x_i,f_i(G(x))\}$, $1\le i < \infty$.
Note $\sum_{i=1}^{\infty}\, h_i(x) >0$ and define
$H:|K'| \to |K|$ by 
\[
H(x) = 
\left( 
\frac{h_i(x)}{\sum_{i=1}^{\infty}\, h_i(x)} 
\right) 
\]
Since $x$ and $H(x)$ as well as $H(x)$ and $F\circ G(x)$ are in the 
same simplex, $1|_{|K|}$ is Lipschitz homotopic to $H$,
and $H$ is Lipschitz homotopic to $F\circ G$.
\end{proof}

\begin{assertion}
$G$ is homotopy equivalent with homotopy inverse $F$. 
\end{assertion}

\begin{proof}
From Assertion \ref{ass:FG}, $G$ induces injective homomorphisms $\pi_{*}(|K|) \to \pi_{*}(M)$ in all
dimensions. 
Note that $M$ has the homotopy type of a $CW$-complex $L$ of finite dimension since $M$ is locally contractible 
and finite dimensional.
For each $m\ge 1$ and each map $\alpha:S^m\to M$,
set $\beta:=G\circ F\circ\alpha:S^m\to M$.
From construction, for each $x\in S^m$, if $F(\alpha(x))$ is contained in an open simplex
$\sigma=(e_{i_0},\ldots,e_{i_k})$, then both $\alpha(x)$ and $\beta(x)$ are contained in 
some $U_{i_j}$ with $0\le j\le k$.
Take a sufficiently fine triangulation $\Sigma$ of $S^m$. If $\alpha(v)$ and $\beta(v)$ is in $U_i$ for a 
vertex $v\in \Sigma^{0}$, we can join $\alpha(v)$ to $\beta(v)$ by a homotopy in $U_i$.
Since $\Sigma$ is sufficiently fine, this homotopy can be extended inductively to 
a homotopy between $\alpha$ and $\beta$ on each skeleton $\Sigma^{\ell}$
with $0\le \ell \le m$.
Namely $\alpha$ is homotopic to $\beta$, and thus $G$ induces isomorphisms 
$\pi_{m}(|K|) \to \pi_{m}(M)$ for all $m$.
Therefore Whitehead's theorem implies that $g$ is homotopy equivalent.
\end{proof}
\end{proof}

\begin{remark} \upshape
In the situation of Theorem \ref{thm:good-cover}, $M$ actually has the same {\it Lipschitz} homotopy type as the 
nerve of any good covering of it. The proof will appear in a forthcoming paper. 
\end{remark}

\section{Stability of good coverings} \label{sec:good-stab}

Let us recall that $\mathcal A(n,D,v_0)$ {denotes} the set of all isometry classes of 
$n$-dimensional compact Alexandrov spaces $M$ with 
curvature $\ge -1$, $\diam(M)\le D$, $\vol(M)\ge v_0>0$.
In this section, we prove Theorem \ref{thm:good-stab}.

{The construction of locally defined strictly concave functions is stable in the non-collapsing convergence, which is stated as follows:
\begin{lemma}[\cite{Kap}, \cite{MY:stab}] \label{lem:stable SLC} 
Let $M \in \mathcal A(n,D,v_0)$ and $M_j \in \mathcal A(n,D,v_0)$ a sequence converging to $M$ as $j \to \infty$. 
Then, for any $p \in M$, there exist $r>0$ satisfying the following: 
there exist a Lipschitz function $\varphi$ (resp. $\varphi_j$) on $M$ (resp. on $M_j$) which is strictly concave on $U(p, r)$ (resp. on $U(p_j,r)$), 
and $p_j \in M_j$ such that 
\begin{itemize}
\item $\varphi_j$ and $p_j$ converge to $\varphi$ and $p$ respectively, under the convergence $M_j \to M$; 
\item $p$ (resp. $p_j$) is the unique maximizer of $\varphi$ (resp. of $\varphi_j$) on $U(p,r)$ (resp. on $U(p_j,r)$).
\end{itemize}
\end{lemma}
}


\begin{proof}[Proof of Theorem \ref{thm:good-stab}]
Let us fix $M \in \mathcal A(n,D,v_0)$. 
We construct a finite good cover 
{$\{U_i\}_{i=1}^N$} of $M$ as in the proof of Theorem \ref{thm:good-cover}(2). 
Let us recall that each $U_i$ has the form $U_i = \{x \in U(p_i,r_i) \mid \varphi_i > 0\}$ for some $p_i \in M$, $r_i > 0$ and a 
{
strictly concave function $\varphi_i$ on $U(p_i,r_i)$ such that $p_i$ is the unique maximizer of $\varphi_i$.}

For $T \subset \{1, \dots, N\}$, we set $U_T = \bigcap_{i \in T} U_i$ and if $U_T$ is nonempty, then we set $m_T := \max_{U_T} \varphi_T$, where $\varphi_T = \min_{i \in T} \varphi_i$ which is strictly concave on $U_T$.
{Furthermore}, we set $m := \min \{m_T \mid U_T \neq \emptyset \}$. 
{We set 
\[
V_i = \{x \in U(p_i,r_i) \mid \varphi_i(x) > \delta_0 m\}. 
\]
Here, $\delta_0 > 0$ is a small number such that 
the family $\{V_i\}_{i=1}^N$ still covers $M$.}
Let $\mathcal U_M = \{V_i \}_{i=1}^N$. 
By {Lemma \ref{lem:SLC},}
$\mathcal U_M$ is a good covering of $M$.

By Lemma \ref{lem:stable SLC}, 
{there is an $\epsilon_0(M) > 0$ depending on $M$}
such that for any $M' \in \mathcal A(n,D,v_0)$ with $d_{\mathrm{GH}}(M,M') < {\epsilon_0(M)}$, there exist $p_i' \in M'$ and 
{a strictly concave function $\varphi_i'$ on $U(p_i',r_i)$ such that $p_i'$	 is the unique maximizer of $\varphi_i'$.}
Here, each $p_i'$ is close to $p_i$ via a Gromov-Hausdorff approximation between $M$ and $M'$. 
{Let} $\psi : M \to M'$ and $\psi' : M' \to M$ {be} $\epsilon$-Gromov-Hausdorff approximations such that 
{
\[
\sup_{x \in M} |\psi' \circ \psi(x), x| < \epsilon, \hspace{1em} \sup_{y \in M'} |\psi \circ \psi'(y), y| < \epsilon
\]
where $\epsilon = 2 d_{\mathrm{GH}}(M,M')$.}
{Furthermore, by Lemma \ref{lem:stable SLC}, we have 
\[
|\varphi_i(x) - \varphi_i'(\psi(x))| < \epsilon_1, \hspace{1em} |\varphi_i(\psi'(y)) - \varphi_i'(y)| < \epsilon_1
\]
for all $x \in M$ and $y \in M'$, where $\epsilon_1 = \epsilon_1(\epsilon) > 0$ satisfies $\lim_{\epsilon \to 0} \epsilon_1(\epsilon) = 0$.} 
{Let us assume that 
\[
\epsilon_1 < \delta_0 m /2 
\]
by taking $\epsilon_0(M)$ to be small.}

{Now we consider the family $\{V_i'\}_{i=1}^N$ defined as 
\[
V_i' := \{x \in U(p_i', r_i) \mid \varphi_i'(x) > \delta_0 m/ 2\}.
\]
}
{We shall prove that $\mathcal U_{M'} = \{V_i'\}_{i=1}^N$ is a good cover of $M'$ and its nerve is isomorphic to that of $\mathcal U_M$.}
{Let us show that $\mathcal U_{M'}$ covers $M'$. 
Indeed, for any $y \in M'$, there exists $i$ such that $\psi'(y) \in V_i$, that is, $\varphi_i(\psi'(y)) > \delta_0 m$. 
Hence, we have 
\begin{align*}
\varphi_i'(y) &> \varphi_i(\psi'(y)) - \epsilon_1 > \delta_0 m - \epsilon_1 > \delta_0 m/2. 
\end{align*}
Therefore, $y \in V_i'$.
So, we have $\bigcup_{i=1}^N V_i' = M'$. 
}
Then, by Lemma \ref{lem:SLC}, the covering $\mathcal U_{M'}$ of $M'$ is good. 
{Next, we show that the nerves of $\mathcal U_{M'}$ and $\mathcal U_M$ are isomorphic.}
For any subset $T \subset \{1, \dots, N\}$, we set $V_T = \bigcap_{i \in T} V_i$ and $V_T' = \bigcap_{i \in T} V_i'$.
By the construction, for every $x \in V_i$, 
{
\[
\varphi_i'(\psi(x)) > \varphi_i(x) - \epsilon_1 > \delta_0 m - \epsilon_1 > \delta_0m /2.
\]
So,} we have $\psi(x) \in V_i'$. 
Hence, if $V_T$ is nonempty, then $V_T'$ is nonempty.
Conversely, let us take $y \in V_T'$. 
Then, 
{\[
\varphi_i(\psi'(y)) > \varphi_i'(y) - \epsilon_1 > 0 
\]
}for all $i \in T$.
Hence, $\psi'(y) \in U_T$. 
{Here, if $\varphi_T(\psi'(y)) > \delta_0 m$, $\psi'(y) \in V_T$. 
So, we may assume that $\varphi_T(\psi'(y)) \le \delta_0 m$.}
By the choice of $m$, along the gradient curve of $\varphi_T$ starting from $\psi'(y)$, we can find a point $x \in U_T$ with $\varphi_T(x) > \delta_0 m$. 
Hence, $V_T$ is nonempty. 
Therefore, the nerves of $\mathcal U_{M}$ and $\mathcal U_{M'}$ are isomorphic
{under the correspondence $V_i \mapsto V_i'$.}


{From the above argument, for every $M \in \mathcal A(n, D,v_0)$, there exists an $\epsilon_M > 0$ satisfying the following: 
there exists a finite simplicial complex $K$ such that every $M' \in U_{\text{GH}}(M,\epsilon_M)$ admits a good covering whose nerve is isomorphic to $K$. 
}
%
Here, $U_{\mathrm{GH}}(M,\epsilon)$ denotes the {open} $\epsilon$-neighborhood of $M$ in $\mathcal A(n,D,{v_0})$. 
{Since $\mathcal A(n,D,v_0)$ is compact, we can take a finitely many Alexandrov spaces $M_1, \dots, M_N \in \mathcal A(n,D,v_0)$ such that $\{U_{\mathrm{GH}}(M_i, \epsilon_{M_i})\}_{i=1}^N$ covers $\mathcal A(n,D,v_0)$.
From the definition of $\epsilon_M$, we obtain simplicial complices $K_1, \dots, K_N$ such that every $M \in U_{\mathrm{GH}}(M_i,\epsilon_{M_i})$ has a good covering whose nerve is isomorphic to $K_i$.
Let $\epsilon_0 > 0$ be the Lebesgue number of the cover $\{U_{\mathrm{GH}}(M_i, \epsilon_{M_i})\}_{i=1}^N$. 
This is the desired constant in the statement of the theorem.}
This completes the proof. 
\end{proof}

\section{Appendiex: a version of fibration theorem} \label{sec:appendix}

Fibration Theorem \ref{thm:fibration} is important to determine the topological structure of Alexandrov spaces via regular functions. 
Actually, Perelman proved it for admissible functions on spaces with or without boundary in \cite{Per Alex II} and \cite{Per}, 
and for general semiconcave functions on spaces without boundary in \cite{Per DC}.
In this section, we prove Theorem \ref{thm:fibration} and its generalization (Theorem \ref{thm:fibration with boundary} stated later), which are versions of fibration theorems for semiconcave functions on spaces with non-empty boundary. 


\subsection{MCS-spaces}
To prove Theorems \ref{thm:fibration} and \ref{thm:fibration with boundary} stated below, we recall the notion of {\it MCS-spaces} introduced by Perelman (\cite{Per Alex II}, \cite{Per}). 
Those spaces are defined inductively: 
$0$-dimensional MCS-spaces are defined to be discrete sets; 
a separable metrizable space $X$ is called an $n$-dimensional MCS-space if for any $x \in X$, there exist an open neighborhood $U$ and a compact $(n-1)$-dimensional MCS-space $Y$ such that $(U,x)$ is homeomorphic to $(K(Y), o)$ as pointed spaces, where $o$ denotes the apex of the cone. 
We call $U$ a conical neighborhood of $x$ and $Y$ a generator of $U$.
Any MCS-space has a canonical stratification into topological manifolds, as follows. 
Let $X$ be an $n$-dimensional MCS-space. 
Then, $X$ has a canonical stratification into topological manifolds $\{X^{(\ell)}\}_{\ell=0}^n$ by the following way: $p \in X$ is in the (canonical) $\ell$-stratum $X^{(\ell)}$ if $p$ has a conical neighborhood homeomorphic to $\mathbb R^\ell \times K$, where $\ell$ is taken to be maximal and $K$ is a cone over some compact MCS-space. 
Then, $X^{(\ell)}$ is a topological $\ell$-manifold. 
We call $X^{(n)}$ the {\it top stratum} of $X$. 
Note that $X^{(n)}$ is always non-empty and dense in $X$. 

Siebenmann proved the following important theorem in the category of topological spaces: 
\begin{theorem}[\cite{Sie}] \label{thm:Siebenmann}
Let $f : X \to Y$ be a proper continuous surjection between topological spaces. 
Suppose that 
\begin{itemize}
\item[(A)] for each $y \in Y$, the fiber $f^{-1}(y)$ is an MCS-space; 
\item[(B)] $f$ is a topological submersion, that is, for any $x \in X$, there exist an open neighborhood $V$ of $x$ and a homeomorphism $\varphi : (f^{-1}(f(x)) \cap V) \times f(V) \to V$ respecting $f$. 
Here, the letter means that $f \circ \varphi$ is the projection of the second coordinate of the product.
\end{itemize}
Then, $f$ is a fiber bundle, that is, for any $y \in Y$, there exist an open neighborhood $U$ of $y$ and a homeomorphism $\theta : f^{-1}(y) \times U \to f^{-1}(U)$ respecting $f$.
\end{theorem}
We call such a map $\varphi$ in the condition (B) a product chart of $f$ at $x$, and a $V$ a product neighborhood.
The theorem actually holds if we replace an MCS-space with a space having some property about deformations of homeomorphisms.
For instance, as in \cite{Sie}, CS sets and WCS sets have such a property. 
MCS-spaces give a middle class of between CS and WCS sets. 

\subsection{The original fibration theorem by Perelman}
Let $\Sigma$ denotes an Alexandrov space of curvature $\ge 1$. 
A semiconcave function $g : \Sigma \to \mathbb R$ is said to be spherically concave if it is $g$-concave, or equivalently, the cone extension $K(g) : K(\Sigma) \to \mathbb R$ is concave, where $K(g)(a \xi) = a g(\xi)$ for $a \ge 0$ and $\xi \in \Sigma$. 
The inner product $\left< g, h \right>$ of two spherically concave functions $g, h$ on $\Sigma$ is defined by 
\[
\left< g, h \right> := \sup_{\xi \in \Sigma} \left(g(\xi)h(\xi) + \left< g_\xi', h_\xi' \right>\right). 
\]
Here, $0$-dimensional Alexandrov spaces are the spaces of single point or consisting of two points of distance $\pi$, and hence, the inner product is actually defined inductively. 
Note that, the derivation of every semiconcave function is spherically concave on the space of directions at each point. 

Let $U$ be an open set of an Alexandrov space $M$ having no boundary points. 
A map $f = (f_1, \dots, f_k)$ consisting of semiconcave functions $f_i$ defined on $U$ is said to be {{\it regular}} at $x \in U$ if 
\begin{itemize}
\item there is $\xi \in \Sigma_x$ such that $(f_i)_x' (\xi) > 0$ for each $i$ and 
\item $\left< (f_i)_y', (f_j)_y' \right> < 0$ for $y$ near $x$ and for every $i \neq j$. 
\end{itemize}
The set of all regular points {of $f$} in $U$ is open. 
The original fibration theorem for semiconcave functions is stated as follows. 
\begin{theorem}[\cite{Per DC}] \label{thm:original fibration}
Let $f = (f_1, \dots, f_k)$ be a map consisting of semiconcave functions on $U$ as above. 
Suppose that $f$ is regular on $U$. 
Then, the following holds. 
\begin{itemize}
\item $f(U)$ is open in $\mathbb R^k$;
\item each fiber of $f$ is an $(n-k)$-dimensional MCS-space;
\item $f : U \to f(U)$ is a topological submersion;
\item further, if $f : U \to f(U)$ is proper, that is, the preimage of every compact set in $f(U)$ is compact, then it is a fiber bundle.  
\end{itemize} 
\end{theorem}

We note that the function $f(x) = -|x|^2 /2$ is $(-1)$-concave on a disk $D = \{x \in \mathbb R^2 \mid |x-1| \le 1 \}$ with the flat metric, but it is not admissible and its extension to the double is not semiconcave.
Indeed, the derivation of $f$ at $x_0=2$ a boundary point is of the form 
\[
f_{x_0}' (\xi) = 2 \cos \angle (0_{x_0}', \xi)
\]
for $\xi \in \Sigma_{x_0}(D)$, which is not contained in the class of DER functions (see \cite{Per}).
Hence, $f$ is not admissible.
Further, every point except $0$ is a regular point of $f$, but the fibration theorem does not hold for $f$, because two regular fibers $f^{-1}(f(x_0)) = \{x_0\}$ and $f^{-1}(f(1))$ are not homeomorphic. 

Perelman's Stability Theorem (\cite{Per Alex II}, \cite{Kap stab}) 
is very important in the geometry of Alexandrov spaces. 
The proof of it is based on the fibration theorem (for admissible functions). 
The fibration Theorems states that each regular fiber is a general MCS-space. 
Further, we can prove that each fiber of dimension one is actually a manifold, by using Stability Theorem as follows. 

\begin{lemma}\label{lem:codim one fiber}
Let $U$ be an open subset in an $n$-dimensional Alexandrov space having no boundary point and $f = (f_i) : U \to \mathbb R^{n-1}$ a map consisting of semiconcave functions $f_i$. 
Suppose that $f$ is regular on $U$. 
Then, the fiber of $f$ at each point in $f(U)$ is a one-manifold without boundary. 
\end{lemma}
\begin{proof}
We may assume that $n \ge 2$. 
The original fibration theorem states that the fiber $F= f^{-1}(v)$ at $v \in f(U)$ is a one-dimensional MCS-space, which is a locally finite graph in general. 
Let $x \in F$ be a vertex of the graph. 
Since $f$ is a topological submersion, $x$ has a conical neighborhood $V$ in $U$ such that 
$(V,x)$ is homeomorphic to $(K(\Lambda) \times \mathbb R^{n-1}, o)$, where 
$\Lambda$ is a link at $x$ in the graph $F$. 
Then, we have 
\[
H^n(U,U \setminus x) \cong \bar H^0(\Lambda). 
\]
Here, the cohomologies are considered having $\mathbb Z_2$-coefficients.
On the other hands, by Stability Theorem and by Grove and Petersen (\cite{GP1}), we have 
\[
H^n(U,U \setminus x) \cong H^{n-1}(\Sigma_x) \cong \mathbb Z_2. 
\]
Therefore, the set $\Lambda$ consists of only two points. 
Consequently, $F$ does not have a branching point, that is, $F$ is a one-manifold without boundary. 
This completes the proof.
\end{proof}

\subsection{Double}
Let $M$ denote an Alexandrov space with nonempty boundary. 
Its double $D(M)$ is defined as the quotient space of the disjoint union $M \amalg M$ by identifying boundary points of two $M$'s. 
Perelman proved that $D(M)$ is also an Alexandrov space with the canonical metric (\cite{Per Alex II}). 
Then, $M$ is regarded as an isometrically embedded subset of $D(M)$. 
For any $x \in D(M)$, we set $r(x)$ the point corresponding to $x$ in the other copy of $M$ in $D(M)$.
This $r$ defines a canonical isometric involution (reflection via the boundary) on $D(M)$, and the fixed point set is equal to $\partial M$.
For any subset $S$ of $M$, we set 
$D(S) = S \cup r(S) \subset D(M)$.
For a map $g : S \to Y$ to a space $Y$, its canonical extension to the double $D(S)$ is defined by $\tilde g(x) = \tilde g(r(x))$ for $x \in S$. 
For $x \in \partial M$, $\tilde B(x, \delta)$ denotes the closed ball in $D(M)$ centered at $x$ of radius $\delta$. 

\subsection{Equivariant incomplementable lemma}
Let $U$ be an open subset of an $n$-dimensional Alexandrov space $M$ such that $U \cap \partial M \neq \emptyset$.  
Let $f : U \to \mathbb R^k$ be a map such that the double extension $\tilde f_i$ is semiconcave for each component $f_i$, where $k<n$. 
Let $p \in U \cap \partial M$ be such that $\tilde f = (\tilde f_i)$ is regular at $p$. 
We say that $\tilde f$ is {\it $r$-complementable at} $p$ if there exist an open neighborhood $V$ of $p$ in $U$ and a Lipschitz function $f_{k+1}$ defined on $V$ such that the extension $\tilde f_{k+1}$ of $f_{k+1}$ is semiconcave on $D(V)$ and $(\tilde f, \tilde f_{k+1})$ is regular at $p$. 
Otherwise, we say that $f$ is {\it $r$-incomplementable at} $p$. 

\begin{lemma} \label{lem:equiv regular nbd}
Let $f, U$ and $M$ be as above. 
Suppose that $\tilde f$ is regular at $p$ and $f$ is $r$-incomplementable at $p$ for $p \in U \cap \partial M$. 
Then, there exist an open neighborhood $V$ at $p$ in $U$, a Lipschitz function $g$ defined on $V$ and a Lipschitz function $H : f(V) \to \mathbb R$ such that $\tilde g$ is semiconcave on $D(V)$; 
\begin{itemize}
\item[(a)] a map $\begin{pmatrix} \mathrm{id}_{\mathbb R^k} & 0 \\ - H & \mathrm{id}_{\mathbb R} \end{pmatrix} : \mathbb R^{k+1} \to \mathbb R^{k+1}$ is bi-Lipschitz on $(f,h)(V)$;  
\item[(b)] setting an $r$-equivariant function $\tilde h := \tilde g - H \circ \tilde f$, we have $\tilde h(p)=0$ and $\tilde h \le 0$ on $D(V)$; 
\item[(c)] $(\tilde f, \tilde g)$ is regular on $\tilde K_\rho \setminus \tilde h^{-1}(0)$; 
\item[(d)] $\tilde f^{-1}(v) \cap \tilde h^{-1}(0) \cap \tilde K_\rho$ is a single-point set for each $v \in \tilde f(\tilde K_\rho)$. 
\end{itemize}
Here, 
$\tilde K_\rho$ is a compact set defined as 
\begin{align*}
\tilde K_\rho &= \{x \in D(V) \mid \|\tilde f(x)-f(p)\| \le \rho, \tilde h(x) \ge - 2 \rho \}
\end{align*}
for small $\rho > 0$. 
Here, the norm on $\mathbb R^k$ is the maximum norm.
\end{lemma}
All the statements follows from the proof of the original corresponding statements in \cite{Per Alex II} and \cite{Per} with some modification. 
\begin{proof}
We are going to show that all objects appeared in the original corresponding statements in \cite{Per}, \cite{Per Alex II} can be constructed in an equivariant way, in our case. 

Since $\tilde f$ is regular at $p$, there is a direction $\xi \in \Sigma_p D(M)$ such that $(\tilde f_i)_p'(\xi) > \epsilon$ for all $i$, where $\epsilon$ is a positive number. 
Note that such a $\xi$ can be assumed to be in $\Sigma_p (\partial M)$. 
Around $p$, we may assume that all $\tilde f_i$ are $\lambda$-concave for some $\lambda > 0$. 
Take a point $q$ in the direction $\xi$, we have 
\[
\tilde f_i(q) > \tilde f_i(p) + \epsilon |pq| + \lambda|pq|^2
\]
for all $i = 1, \dots, k$. 
We may assume that $q \in \partial M$. 
For some $\delta > 0$ smaller than a constant depending on $\epsilon, \lambda, |pq|$, we have 
\[
\tilde f_i(y) > \tilde f_i(x) + (\epsilon/4) |xy| + (\lambda/2) |xy|^2
\]
for all $x \in \tilde B(p, \delta)$ and $y \in \tilde B(q, \epsilon|pq|/4)$.
For a fine discrete set $\{q_\alpha\}_{\alpha=1}^N \subset \tilde B(q, \epsilon|pq|/4) \cap \partial \tilde B(p,|pq|)$, the function 
\[
\tilde g = \frac{1}{N} \sum_{\alpha=1}^{N} \chi (|q_\alpha \cdot|)
\]
is strictly concave on $\tilde B(p,\delta)$, for some concave increasing real-to-real function $\chi$ which is strictly concave near $|pq|$. 
Now, we should note that $\{q_\alpha\}$ can be taken to be $r$-invariant. 
Hence, the function $\tilde g$ is $r$-equivariant. 
Let us set $V := B(p, \delta)$ and $g$ the restriction of $\tilde g$ to $V$.
We define a function $H : f(V) \to \mathbb R$ by $H(v) = \max \{g(x) \mid x \in V \cap f^{-1}(v) \} = \max \{\tilde g(x)\mid x \in D(V) \cap \tilde f^{-1}(v)\}$.  
Then, a function
\[
\tilde h = \tilde g - H \circ \tilde f
\]
is $r$-equivariant. 
From the proof of the original statements, all properties (a), (b), (c) and (d) hold, for constructed $V,g$ and $H$ in our case.
We leave the complete proof to readers. 
\end{proof}
We give several immediate consequences of Lemma \ref{lem:equiv regular nbd}.

By the property (c) in Lemma \ref{lem:equiv regular nbd} and the compactness of $\tilde K_\rho${, and by the original Fibration Theorem \ref{thm:original fibration}, the map $(\tilde h, \tilde f) : \tilde K_\rho \setminus \tilde h^{-1}(0) \to [-2 \rho,0) \times D^k(\rho)$ is a fiber bundle.}
Here, 
$D^k(\rho)$ is given as 
\begin{align*}
D^k(\rho) &= \tilde f(\tilde K_\rho) = \{v \in \mathbb R^k \mid \|v-f(p)\| \le \rho\}.
\end{align*}
Since the image $[-2 \rho, 0) \times D^k(\rho)$ is contractible, there is a homeomorphism 
\begin{equation} \label{eq:trivialization}
\varphi : \tilde K_\rho \setminus \tilde h^{-1}(0) \to \tilde \Pi_\rho \times [-2\rho, 0) \times D^k(\rho)
\end{equation}
respecting $(\tilde h, \tilde f)$. 
Here, $\tilde \Pi_\rho$ is a regular fiber of $(\tilde h, \tilde f)$ in $\tilde K_\rho$ given as 
\[
\tilde \Pi_\rho = \tilde f^{-1}(f(p)) \cap \tilde h^{-1}(-\rho) \cap \tilde K_\rho
\]
which is an $(n-k)$-dimensional compact MCS-space, by the original fibration theorem (\cite{Per DC}). 
The set $\tilde \Pi_\rho$ is $r$-invariant, by the definition. 

We prepare some symbols which will be used later. 
Let us set 
$K_\rho = \tilde K_\rho \cap U$, 
$\Pi_\rho = \tilde \Pi_\rho \cap U$, and  
$h = g - H \circ f$. 

Note that by (d), we have $\tilde K_\rho \cap \tilde h^{-1}(0) \subset \partial M$. 
Indeed, the set $\tilde f^{-1}(v) \cap \tilde K_\rho \cap \tilde h^{-1}(0)$ is $r$-invariant, by the definition, and is a single-point set by (d). 
Therefore, it is contained in $\partial M$. 
Remark that, $\tilde K_\rho \cap \tilde h^{-1}(0)$ does not coincide with $\tilde K_\rho \cap \partial M$, in general. 

\subsection{A fibration theorem in our case}
From now on, 
we fix the following situation. 
Let $f_1, \dots, f_k$ be locally Lipschitz functions defined on an open subset $U$ of an $n$-dimensional Alexandrov space $M$ with boundary. 
We consider the case $U \cap \partial M \neq \emptyset$ and the canonical extension $\tilde f_i$ on $D(U)$ is semiconcave for every $i$. 

\begin{lemma} \label{lem:k<n}
Let $f = (f_i)$, $U$ and $M$ be as above.  
Suppose that $\tilde f = (\tilde f_i)$ is regular at some $p \in D(U) \cap \partial M$. 
Then, we have $k \le n-1$.
Further, if $\tilde f$ is regular on $U$, then $f: U \to \mathbb R^k$ is an open map.   
\end{lemma}
\begin{proof}
Since $\tilde f$ is regular at $p$, by \cite{Per DC}, we have $k \le n$. 
We suppose $k=n$. 
Let $p \in U \cap \partial M$. 
The original fibration theorem states that $\tilde f$ is homeomorphic near $p$ in $D(U)$. 
That is, there is a neighborhood $V$ of $p$ in $U$ such that $\tilde f : D(V) \to \mathbb R^n$ is an embedding. 
The restriction $f : V \to \mathbb R^n$ is also an embedding. 
However, because $\tilde f$ is $r$-equivariant, the images $f(V)$ and $\tilde f(D(V))$ coincide. 
It is a contradiction.
Hence, we have $k \le n-1$. 

Suppose that $\tilde f$ is regular on $D(U)$. 
By Theorem \ref{thm:original fibration}, the map $\tilde f : D(U) \to \mathbb R^k$ is an open map. 
For any open set $O$ in $U$, its double $D(O)$ is open in $D(U)$. 
Therefore, the map $f : U \to f(U)$ is also an open map. 
Hence, the letter conclusion is proved. 
\end{proof}

\begin{theorem}\label{thm:fibration with boundary}
Let $f= (f_i)$, $U$ and $M$ be as above.  
Suppose that $\tilde f = (\tilde f_i)$ is regular on $D(U)$. 
Then, we have 
\begin{itemize}
\item[(A)] For every $v \in f(U)$, its fiber $f^{-1}(v)$ is an $(n-k)$-dimensional MCS-space. 
\item[(B)] $f : U \to f(U)$ is a topological submersion. 
\item[(C)] If $f : U \to f(U)$ is proper, then it is a fiber bundle. 
\end{itemize}
\end{theorem}
\begin{proof}
The properties (A), (B) and (C) for $k$, where $k$ is the dimension of the target of $f$, are denoted by (A)$_k$, (B)$_k$ and (C)$_k$.
We prove the properties by the backward induction on $k$ as the proof of the original fibration theorem. 
By Theorem \ref{thm:Siebenmann}, if (A)$_k$ and (B)$_k$ hold, then (C)$_k$ holds. 


Let us prove (A)$_{n-1}$ and (B)$_{n-1}$. 
To prove them, we find a product neighborhood at a point $p \in U \cap \partial M$ with respect to $f$. 
Note that every point in $U \setminus \partial M$ already has a product neighborhood, by the original fibration theorem (\cite{Per DC}). 
By Lemma \ref{lem:k<n}, $f$ is $r$-incomplementable at $p$. 
Then, there exist an open neighborhood $V$ of $p$ in $U$, a Lipschitz function $g$ defined on $V$ and a Lipschitz function $H : f(V) \to \mathbb R$ satisfying the conclusion of Lemma \ref{lem:equiv regular nbd}. 
By using them, we have a product chart at $p$ in $\tilde K_\rho$ with respect to $(\tilde f, \tilde h)$ as the following way. 
First, by properties (a) and (c) in Lemma \ref{lem:equiv regular nbd} and by the original fibration theorem, we obtain a homeomorphism 
\[
\varphi : \tilde K_\rho \setminus \tilde h^{-1}(0) \to \tilde \Pi_\rho \times [-2 \rho, 0) \times D^{n-1}(\rho)
\]
respecting $(\tilde h, \tilde f)$ as in \eqref{eq:trivialization}. 
By the property (d), we have a canonical extension $\psi$ of $\varphi$ 
\[
\psi : \tilde K_\rho \to \bar K(\tilde \Pi_\rho) \times D^{n-1}(\rho)
\]
which is a homeomorphism respecting $\tilde f$. 
Let us set $\tilde F_v = \tilde f^{-1}(v) \cap \tilde K_\rho$ and denote by $p(v)$ the unique point contained in $\tilde F_v \cap \tilde h^{-1}(0)$ for $v \in D^{n-1}(\rho)$. 
In particular, $p(f(p)) = p$. 
Then, we have $(\tilde F_v, p(v))$ is homeomorphic to $(\bar K(\tilde \Pi_\rho), o)$. 
Since the relative interior of the fiber $\tilde F_v$ is a one-manifold without boundary, $\tilde \Pi_\rho$ must be a two-points set. 
Because $\tilde \Pi_\rho$ is $r$-invariant, the set $\Pi_\rho$
consists of only one point. 
We observe that $\tilde F_v \cap \tilde h^{-1}(0) = \tilde F_v \cap \partial M$. 
Indeed, by the remark after Lemma \ref{lem:equiv regular nbd}, we know that $\tilde F_v \cap \tilde h^{-1}(0) \subset \partial M$.
We suppose that $\tilde F_v$ meets $\partial M$ at least two points. 
Then, since $\tilde F_v$ is $r$-invariant, $\tilde F_v$ contains a circle, which contradicts to that $\tilde F_v$ is an interval. 
Therefore, we have $\tilde F_v \cap \tilde h^{-1}(0) = \tilde F_v \cap \partial M$.
Let us set $F_v = f^{-1}(v) \cap K_\rho$. 
We conclude that $(\tilde F_v, F_v, p(v))$ is 
homeomorphic to $([-1,1], [0,1], 0)$ as triples of spaces.  
In particular, 
(A)$_{n-1}$ holds. 
Further, we know that 
the restriction of $\varphi$ to $K_\rho \setminus \tilde h^{-1}(0) = K_\rho \setminus \partial M$ has the image $\Pi_\rho \times [-2 \rho, 0) \times D^{n-1}(\rho)$. 
Therefore, the restriction of $\psi$ to $K_\rho$ is a homeomorphism with the target $\bar K(\Pi_\rho) \times D^{n-1}(\rho)$ respecting $f$. 
It gives a product chart at $p$ of $f$. 
This completes the proof of (B)$_{n-1}$. 


Next, we are going to prove (A)$_k$ and (B)$_k$ for $k < n-1$, assuming (A)$_{k+1}$, (B)$_{k+1}$ and (C)$_{k+1}$. 
Let $F = f^{-1}(v)$ for $v \in f(U)$. 
We already know that every point in $F \setminus \partial M$ has a conical neighborhood as an MCS-space, by the original fibration theorem (\cite{Per DC}). 
We may assume that $F \cap \partial M$ is not the empty-set, and take a point $p$ in the intersection. 
To prove (A)$_k$ and (B)$_k$, we find a conical neighborhood at $p$ in $F$ as an MCS-space and a product neighborhood at $p$ with respect to $f$. 
If $f$ is $r$-complementable at $p$, then there is a function $f_{k+1}$ defined near $p$ such that $\tilde f_{k+1}$ is semiconcave and $(\tilde f, \tilde f_{k+1})$ is regular at $p$ in the double of a neighborhood of $p$. 
Then, by (B)$_{k+1}$, we have a product chart 
\[
\theta : V \to (f^{-1}(f(p)) \cap f_{k+1}^{-1}(f_{k+1}(p)) \cap V) \times (f,f_{k+1})(V)
\]
at $p$ of $(f,f_{k+1})$. 
Taking $V$ to be small, we may assume that the image of $(f,f_{k+1})$ is the form $(f,f_{k+1})(V) = f(V) \times (a,b)$, where $f_{k+1}(V) = (a,b)$. 
Let us set $F := f^{-1}(f(p)) \cap f_{k+1}^{-1}(f_{k+1}(p)) \cap V$, which is an $(n-k-1)$-dimensional MCS-space. 
We obtain a product chart 
\[
\eta : V \to (f^{-1}(f(p)) \cap V) \times f(V)
\]
of $f$. 
Then, $f^{-1}(f(p)) \cap V$ is homoeomorphic to $F \times (a,b)$, which is an $(n-k)$-dimensional 
MCS-space. 
Therefore, in this case, (A)$_k$ and (B)$_k$ are proved. 

We next assume that $f$ is $r$-incomplementable at $p$. 
Then, there exist a neighborhood $V$ of $p$, a function $g$ defined on $V$ and a function $H : f(V) \to \mathbb R$ satisfying the conclusion of Lemma \ref{lem:equiv regular nbd}.  
Since $(\tilde f, \tilde g)$ is regular on $\tilde K_\rho \setminus \tilde h^{-1}(0)$, by (C)$_{k+1}$, we have a homeomorphism 
\[
\varphi : K_\rho \setminus \tilde h^{-1}(0) \to \Pi_\rho \times [-2 \rho, 0) \times D^k(\rho)
\]
respecting $(h,f)$. 
The space $\Pi_\rho$ is an $(n-k-1)$-dimensional MCS-space, by (A)$_{k+1}$. 
By (d) in Lemma \ref{lem:equiv regular nbd}, $\varphi$ has a canonical extension 
\[
\psi : K_\rho \to \bar K(\Pi_\rho) \times D^k(\rho)
\]
which is a homeomorphism respecting $f$. 
Then, it gives a product chart at $p$ with respect to $f$. 
This implies (B)$_k$. 
Further, by the construction of $\psi$, we have that $(f^{-1}(f(p)) \cap K_\rho, p)$ is homeomorphic to $(\bar K(\Pi_\rho), o)$. 
Therefore, $f^{-1}(f(p))$ is an $(n-k)$-dimensional MCS-space. 
Hence, (A)$_k$ is proved. 
This completes the proof of Theorem \ref{thm:fibration with boundary}.
\end{proof}

\begin{remark} \upshape
As Lemma \ref{lem:codim one fiber}, we can prove that each fiber in Fibration Theorems \ref{thm:original fibration} and \ref{thm:fibration with boundary} belongs to some restricted class of MCS-spaces, if $k$ is general. 
For instance, it is represented as a non-branching MCS-space introduced in \cite{HS}. 
The proof will appear in a forthcoming paper. 
\end{remark}


\end{document}